\documentclass[11pt]{article}
\usepackage{amsmath}
\usepackage{amsfonts}
\usepackage{amssymb}
\usepackage{enumerate}
\usepackage{amsthm}
\usepackage{authblk}
\usepackage{algorithm}
\usepackage{graphicx}
\usepackage{verbatim}
\usepackage{multirow}

\newtheorem{myproposition}{Proposition}[section]
\newtheorem{mytheorem}[myproposition]{Theorem}
\newtheorem{mylemma}[myproposition]{Lemma}

\newtheorem{mycorollary}[myproposition]{Corollary}
\newtheorem{myfact}[myproposition]{Fact}

\newtheorem{myobservation}[myproposition]{Observation}

\makeatletter
\def\imod#1{\allowbreak\mkern10mu({\operator@font mod}\,\,#1)}
\makeatother

\begin{document}

\title{Deriving Priorities From Inconsistent PCM using the Network Algorithms}

\author[1]{Marcin Anholcer}
\author[2]{Janos F\"ul\"op}
\affil[1]{\scriptsize{}Pozna\'n University of Economics, Faculty of Informatics and Electronic Economy}
\affil[ ]{Al. Niepodleg{\l}o\'sci 10, 61-875 Pozna\'n, Poland, \textit{m.anholcer@ue.poznan.pl}}
\affil[ ]{}
\affil[2]{Research Group of Operations Research and Decision Systems}
\affil[ ]{Computer and Automation Research Institute, Hungarian Academy of Sciences}
\affil[ ]{P.O. Box 63, 1518 Budapest, Hungary, \textit{fulop@sztaki.hu}}

\date{}

\maketitle

\begin{abstract}
In several multiobjective decision problems Pairwise Comparison Matrices (PCM) are applied to evaluate the decision variants. The problem that arises very often is the inconsistency of a given PCM. In such a situation it is important to approximate the PCM with a consistent one. The most common way is to minimize the Euclidean distance between the matrices. In the paper we consider the problem of minimizing the maximum distance. After applying the logarithmic transformation we are able to formulate the obtained subproblem as a Shortest Path Problem and solve it more efficiently. We analyze and completely characterize the form of the set of optimal solutions and provide an algorithm that results in a unique, Pareto-efficient solution.
\end{abstract}

\section{Introduction}

One of the commonly used tools of multiobjective decision making is the Analytic Hierarchy Process, introduced by Saaty (see e.g. \cite{ref_saa1980,ref_erk1991}) and studied by numerous authors. During the process, the Decision Maker compares pairwise the given decision variants. His preferences are defined using a so-called pairwise comparison matrix $A = [a_{ij}]$, where the positive number $a_{ij}$ says how many times the variant $i$ is better than the variant $j$. In order to make the preferences consistent, the values of $a_{ij}$, $i = 1, 2,\dots, n$, $j = 1, 2, \dots, n$ should satisfy the following conditions:

\begin{eqnarray}\label{cond1}
a_{ji}=\frac{1}{a_{ij}} &\text{for}& i,j = 1, 2, \dots, n,\\
a_{ij}a_{jk}  = a_{ik}  &\text{for}& i,j,k=1,2,\dots,n.\label{cond2}
\end{eqnarray}

The matrix $A$ defining consistent preferences is also called consistent. The condition (\ref{cond1}) is usually satisfied, since one may ask the Decision Maker to fill only the part of matrix over the diagonal and then calculate the remaining elements as the reciprocals of the introduced numbers. The condition (\ref{cond2}) is unfortunately more difficult to fulfill and is the usually the source of the inconsistency.

It can be proved (see e.g. \cite{ref_saa1980}) that $A$ is consistent if and only if there exists a vector $v=(v_1, v_2, \dots, v_n)$ with positive entries such that

\begin{eqnarray}\label{formula_w_A}
a_{ij}=\frac{v_i}{v_j} &\text{for} &  i,j=1,2,\dots, n.
\end{eqnarray}

The elements of $v$ are interpreted as the weights expressing the priorities of the decision variants. Finding their values is essential for analyzing the preferences of the Decision Maker. Of course if some vector $v$ satisfies (\ref{formula_w_A}), then so does the vector $\lambda v$ for every $\lambda>0$.

\section{Problem formulation}

In the real life problems the matrix $A$ is usually inconsistent, so it is impossible to find any vector $v$ satisfying (\ref{formula_w_A}), as it does not exist. The goal is to find the vector $v$ that defines the matrix $B$ of the entries $v_i/v_j$ which is the closest one to $A$ in some sense. In addition to this approximation problem, we are also interested in finding a (Pareto) efficient solution of the following vector optimization problem (see e.g. \cite{ref_BlaCarCon}):

\begin{equation}
\min_{v\in R^n_{++}}\left(\left|a_{ij}-\frac{v_i}{v_j} \right|\right).\label{vector_problem}
\end{equation}

There are many methods for solving the problem of approximating the matrix $A$ by a consistent matrix $B$ generated by a suitable vector $v$ of weights. Saaty proposed the principal eigenvector of $A$ for using as $v$. He also introduced an inconsistency index based on the maximal eigenvalue. See \cite{ref_saa1980} for the details.

Another approach is to minimize the average error of the approximation. One of the most popular measures is the square mean calculated according to the formula

\begin{eqnarray}\label{formula_G2}
G_2(A,v)=\left[\frac{1}{n^2}\sum_{i=1}^n\sum_{j=1}^n\left(a_{ij}-\frac{v_i}{v_j} \right)^2 \right]^{1/2},
\end{eqnarray}
and its minimization is equivalent to the Least Squares Method (LSM) \cite{ref_boz2008, ref_ful2008}.

Beside the one proposed in \cite{ref_saa1980}, further methods of measuring the inconsistency were introduced and analyzed, e.g., in \cite{ref_chu1979, ref_koc1993, ref_anh2010, ref_boz2008, ref_ful2010, ref_boz2008a, ref_mog2009, ref_bru2014}. For a positive matrix $A$, a statistical and axiomatic approach leading to the  geometric means solution were used in \cite{ref_hov2008}. Some simulation experiments comparing different inconsistency measures have been performed by \cite{ref_mog2009}. In the latter paper the authors also considered the case of the general mean which can be defined, for $p>0$, as

\begin{eqnarray}\label{formula_Gp}
G_p(A,v)=\left[\frac{1}{n^2}\sum_{i=1}^n\sum_{j=1}^n\left| a_{ij}-\frac{v_i}{v_j} \right| ^p \right]^{1/p}.
\end{eqnarray}

Beside the case $p=2$ mentioned above, the cases $p=1$ and $p\to\infty$ leading to the arithmetic mean and the maximum norm, respectively, seem to be the most important, but further cases are also considered in \cite{ref_mog2009}.

A survey of the methods of deriving the priorities can be found in \cite{ref_ChoWed}. The authors compare the performance of almost 20 different methods. Also, a shorter survey can be found in \cite{ref_Lin}. Here the author studies in particular the preference weighted least square (PWLS) method, which leads to a convex programming optimization problem (more general version of this method is also studied in \cite{ref_DopGon}). The convexity of the optimization problems is not always guaranteed. For example, LSM may lead to nonconvex problems, which makes the solving process more elaborate and time-consuming. Standard global optimization methods were proposed in \cite{ref_CarMes} and \cite{ref_ful2008} for solving these difficult, multi-modal problems. See also \cite{ref_boz2008} for another approach.

In \cite{ref_BlaCarCon} a framework was presented  for testing the ability of the methods to provide a Pareto-optimal solution for (\ref{vector_problem}). In particular it is shown that the principal eigenvector method is not always effective, namely,  in case of some PC matrices it results in a solution which is not Pareto-optimal. See also
\cite{ref_boz2014} for further achievements on this topic. In this paper we study a method that always results in a Pareto-optimal solution.

Our method concerns  the $G_\infty$ variant of the generalized mean, where, for the sake of simplicity, we omit the constant $1/n^2$ from the objective function. In \cite{ref_ChoWed} this measure and the corresponding method is called as \textit{least worst absolute error} (LWAE). To be more specific, we want to solve the following problem:

\begin{eqnarray}\label{problem11}
\min{ G_\infty (A,v)=\max_{1\leq i,j\leq n}{\left| a_{ij}-\frac{v_i}{v_j} \right|} },\\
v_1=1,\\						
v_j>0,& j=1,2,\dots,n.\label{problem12}
\end{eqnarray}

The constraint $v_1=1$ has been introduced in order to normalize the vector $v$. If some $v$ is the solution to the above problem, then so does $\lambda v$ is for every $\lambda>0$, yielding the same objective function value. Other normalizing constraints have been used in other papers, see e.g. \cite{ref_anh2010,ref_boz2008,ref_ful2008}.

In the form of (\ref{problem11})-(\ref{problem12}), the problem  seems to be a difficult optimization problem since the objective function is not convex, thus no local search algorithm may be applied in order to find the global optimum. However, some transformations allow us to solve this problem efficiently. In Section \ref{section_alg} we provide a new solution method, using the logarithmic transformation on the problem and a network algorithm, as well as a root finding algorithm. In Section \ref{section_unique} we characterize the set of optimal solutions and give sufficient and necessary conditions for the optimum to be unique. Moreover, we provide and justify a method for comparing the optimal solutions, and provide an algorithm that results in a unique solution. Although not every solution of (\ref{problem11})-(\ref{problem12}) must be efficient, we prove that in the set of optimal solutions there is always a Pareto-optimal one, which is found by our method.

We finish the paper with computational experiments and conclusions.

\section{New algorithm for the LWAE problem}\label{section_alg}

A solution method to the problem (\ref{problem11})-(\ref{problem12}) was proposed in \cite{ref_anh2012}. The problem has been reformulated as follows, using an additional variable $z$ replacing $G_\infty(A,v)$:
\begin{eqnarray}\label{problem21}
\min{z}\\
\left| a_{ij}-\frac{v_i}{v_j}\right|\leq z,& i,j=1,2,\dots,n\\
v_1=1, \\
v_j>0,& j=1,2,\dots,n.\label{problem22}
\end{eqnarray}

Then the parametrization of $z$ has been performed. The main algorithm has the form of a bisection method on the variable $z$, while the subproblem for given $z$, solved at every step, is the following linear programming problem. In \cite{ref_anh2012} the modified simplex method was used to solve it. The problem under consideration was used to find any feasible solution of the linear system
\begin{eqnarray}
-v_i+(a_{ij}-z) v_j\leq 0,\ 1\leq i<j\leq n,\label{new_system1}\\
-v_i+\frac{v_j}{a_{ji}+z}\leq 0,\  1\leq i<j\leq n,\label{new_system2}\\
\frac{v_i}{a_{ij}+z}-v_j\leq 0,\  1\leq i<j\leq n,\label{new_system3}\\
(a_{ji}-z) v_i-v_j\leq 0,\  1\leq i<j\leq n,\label{new_system4}\\
v_j> 0,\  j=1,\dots, n,\label{new_system5}\\
v_1=1,\label{new_system6}
\end{eqnarray}

\noindent{}or to determine that it has no solution. However, using the similar transformation as in \cite{ref_ful2008}, we are able to find an instance of the Shortest Path Problem that may be used for the same purpose. The problem (\ref{new_system1})-(\ref{new_system6}) can be then rewritten as follows:
\begin{eqnarray}
\frac{v_i}{v_j}\geq(a_{ij}-z),\ 1\leq i<j\leq n,\label{new2_system1}\\
\frac{v_i}{v_j}\geq\frac{1}{a_{ji}+z},\ 1\leq i<j\leq n,\label{new2_system2}\\
\frac{v_j}{v_i}\geq\frac{1}{a_{ij}+z},\ 1\leq i<j\leq n,\label{new2_system3}\\
\frac{v_j}{v_i}\geq(a_{ji}-z),\ 1\leq i<j\leq n,\label{new2_system4}\\
v_j> 0,\ j=1,\dots, n,\label{new2_system5}\\
v_1=1.\label{new2_system6}
\end{eqnarray}

We use the following substitution:
\begin{eqnarray}
w_j=\ln{v_j},\ j=1,\dots ,n.
\end{eqnarray}

In this way we can reformulate (\ref{new2_system1}-(\ref{new2_system6}) as follows:

\begin{eqnarray}
w_j-w_i\leq-\ln\max\{a_{ij}-z,\frac{1}{a_{ji}+z}\}=l_{ij},\ 1\leq i<j\leq n,\label{log_system1}\\
w_i-w_j\leq-\ln\max\{a_{ji}-z,\frac{1}{a_{ij}+z}\}=l_{ji},\ 1\leq i<j\leq n,\label{log_system2}\\
w_1=0.\label{log_system3}
\end{eqnarray}

We consider the version of the Shortest Path Problem, where the lengths of the arcs can be negative. Assume that the underlying network $N=N(z)$ is a complete digraph of the nodes $\{1,\dots,n\}$, the length of the arc $(i,j)$ is equal to $l_{ij}$ and we are looking for the shortest paths from the node $1$ to all the remaining nodes.

\begin{myproposition}\label{neg_cycle}
Two following statements are true.
\begin{enumerate}
\item
If there is no negative cycle in the network $N$, then let $d_j$ denote the shortest distance from node $1$ to node $j$. The vector of the components
\begin{eqnarray}
w_j=d_j,\ j=1,\dots ,n,
\end{eqnarray}
is a solution to the system (\ref{log_system1})-(\ref{log_system3}).
\item
If there exists a negative cycle in $N$, then the system (\ref{log_system1})-(\ref{log_system3}) is inconsistent.
\end{enumerate}
\end{myproposition}

\begin{proof}
Assume there is no negative cycle in $N$. Let $dist(i,j)$ denote the shortest distance between the nodes $i$ and $j$. From the triangle inequality for the nodes $1$, $i$ and $j$ we have
$$
dist(1,i)+dist(i,j)\geq dist(1,j).
$$
But $dist(1,i)=d_i$ for every $i$ and $dist(i,j)\leq l_{ij}$ for every $i,j$, so
$$
d_i+l_{ij}\geq d_j
$$
and analogously
$$
d_j+l_{ji}\geq d_i.
$$
Substituing $i$ with $1$ in the triangle inequality we obtain
$$
d_1\geq dist(1,j)-dist(1,j)=0,
$$
and on the other hand
$$
d_1\leq l_{11}=0,
$$
so all the constraints (\ref{log_system1})-(\ref{log_system3}) are satisfied.

Assume now there is a cycle $(i_1,i_2,\dots,i_s,i_1)$ in $N$ with length $L<0$. Consider the system of inequalities, being the subsystem of (\ref{log_system1})-(\ref{log_system3}):
$$
\begin{array}{l}
w_{i_2}-w_{i_1}\leq l_{i_1,i_2},\\
w_{i_3}-w_{i_2}\leq l_{i_2,i_3},\\
\quad\quad\quad\vdots\\
w_{i_s}-w_{i_{s-1}}\leq l_{i_{s-1},i_s},\\
w_{i_1}-w_{i_s}\leq l_{i_s,i_1}.
\end{array}
$$
Adding the inequalities we obtain
$$
0\leq L,
$$
a contradiction. Thus in such a case the system (\ref{log_system1})-(\ref{log_system3}) is inconsistent.
\end{proof}

A survey of the methods of finding negative cycles in networks can be found in \cite{ref_CheGol}. The instance of Shortest Path Problem discussed above will be solved with the special version of Network Simplex Algorithm, described in \cite{ref_AhuMagOrl}, pp.\   425-430. The algorithm handles negative arc lengths, returns the distances from source to all the nodes based on the node potentials and easily identifies a negative cycle if such a cycle exists. Thus this method allows to identify which part of the Proposition \ref{neg_cycle} is to be used.

Let us consider the length of the arc $(i,j)$ in (\ref{log_system1}) as a function of $z$:
$$l_{ij}(z)=-\ln\max\{a_{ij}-z,\frac{1}{a_{ji}+z}\}.$$

The following observation will be useful in our further considerations.

\begin{myproposition}\label{lij_properties}
For every $i$ and $j$ the function $l_{ij}(z)$ has the following properties:
\begin{enumerate}
\item
$l_{ij}(z)$ is a strictly increasing continuous function of $z$ over $[0,\infty)$.
\item
If $a_{ij}\le 1$, then
$$l_{ij}(z)=\ln (a_{ji}+z)$$
is a strictly concave function over $[0,\infty)$.
\item
If $a_{ij}> 1$, then at $z=a_{ij}-a_{ji}$,  $l_{ij}$ has a single inflexion point over $[0,\infty)$,
$$l_{ij}(z)=-\ln (a_{ij}-z)$$
is strictly convex over $[0,a_{ij}-a_{ji}]$, and
$$l_{ij}(z)=\ln (a_{ji}+z)$$
is strictly concave over $[a_{ij}-a_{ji},\infty)$.
\end{enumerate}
\end{myproposition}

\begin{proof}
Observe that for $z\geq 0$ we have
$$
l_{ij}(z)=\left\{
\begin{array}{lll}
-\ln(a_{ij}-z), & \text{if} & 0\leq z\leq a_{ij}-a_{ji},\\
-\ln(\frac{1}{a_{ji}+z})=\ln(a_{ji}+z), & \text{if} & z> a_{ij}-a_{ji}.
\end{array}
\right.
$$

Obviously $l_{ij}(z)$ is continuous on both intervals. Moreover, we have
$$
\lim_{z\rightarrow (a_{ij}-a_{ji})^-}{l_{ij}(z)}=\lim_{z\rightarrow (a_{ij}-a_{ji})^+}{l_{ij}(z)}=l_{ij}(a_{ij}-a_{ji})=\ln(a_{ij}),
$$
so the function is continuous over $[0,\infty]$. Assuming that $z>0$, the first derivative equals to
$$
l^\prime_{ij}(z)=\left\{
\begin{array}{lll}
\frac{1}{a_{ij}-z}>0, & \text{if} & z< a_{ij}-a_{ji},\\
\frac{1}{a_{ji}+z}>0, & \text{if} & z> a_{ij}-a_{ji}.
\end{array}
\right.
$$
Moreover, we have
$$
\lim_{z\rightarrow 0^+}{l^\prime_{ij}(z)}\in \{a_{ij},a_{ji}\},\\
\lim_{z\rightarrow (a_{ij}-a_{ji})^-}{l^\prime_{ij}(z)}=a_{ij}>0,\\
\lim_{z\rightarrow (a_{ij}-a_{ji})^+}{l^\prime_{ij}(z)}=a_{ji}>0,
$$
so $l_{ij}(z)$ is increasing on whole interval $[0,\infty]$. Finally, the second derivative equals to
$$
l^{\prime\prime}_{ij}(z)=\left\{
\begin{array}{lll}
\frac{1}{(a_{ij}-z)^2}>0, & \text{if} & z< a_{ij}-a_{ji},\\
\frac{-1}{(a_{ji}+z)^2}<0, & \text{if} & z> a_{ij}-a_{ji},
\end{array}
\right.
$$
and thus $l_{ij}(z)$ is convex for $z\le a_{ij}-a_{ji}$ and concave for $z\ge a_{ij}-a_{ji}$.
\end{proof}

For a cycle $C$, let $l_C(z)$ be the sum of the lengths of the edges of $C$. From the Proposition \ref{neg_cycle} it follows that a cycle $C$ in $N$ such that $l_C(z)<0$ exists if and only if the system (\ref{log_system1})-(\ref{log_system3}) is inconsistent. This property can be exploited as follows. Assume that the optimal solution to the problem (\ref{problem21})-(\ref{problem22}) lies in some interval $[z_{\min},z_{\max}]$, where (\ref{log_system1})-(\ref{log_system3}) is inconsistent for $z_{\min}$ and has a feasible solution for $z_{\max}$. Then the following is true.

\begin{myproposition}
Let $[z_{\min},z_{\max}]$ be the actual interval of search, i.e.\ the optimal solution to the problem (\ref{problem21})-(\ref{problem22}) lies in the interval $[z_{\min},z_{\max}]$, where (\ref{log_system1})-(\ref{log_system3}) is inconsistent for $z_{\min}$ and has a feasible solution for $z_{\max}$. Let $C$ be a cycle such that $l_C(z_{\min})<0$. Let $\bar z\in(z_{min},z_{max}]$. If $l_C(\bar z)<0$, then the problem (\ref{log_system1})-(\ref{log_system3}) is inconsistent for every $z\in[z_{min},\bar{z}]$.
\end{myproposition}

\begin{proof}
As $l_{ij}(z)$ is strictly increasing for every arc $(i,j)$ (Proposition \ref{lij_properties}), then also $l_C(z)$ is. So if $l_C(\bar{z})<0$ for some $\bar{z}$, then also $l_C(z)<0$ for every $z<\bar{z}$. Finally, by the Proposition \ref{neg_cycle}, the system is inconsistent for every $z<\bar{z}$.
\end{proof}

The immediate consequence of the above proposition is that we are able to shorten the search interval in a very simple way. Assume that given a cycle $C$ and a nonnegative real $z$, we have $l_C(z)<0$. Since $l_C$ is a strictly increasing function of $z$, it is easy to find a $\bar z >z$ such that $l_C(\bar z)= 0$. Now we can proceed in two ways.

First idea is to use the bisection method.

The second possibility is to check by solving a Shortest Path Problem whether (\ref{log_system1})-(\ref{log_system3}) is consistent with $z=\bar z$. If it is so, we are done: $z=\bar z$ is an optimal solution. Otherwise, the algorithm for solving SPP returns a cycle $\bar C$ such that $l_{\bar C}(\bar z)<0$, and we repeat the above process with $z\gets \bar z$ and $C\gets \bar C$.

We are going to describe both suggested methods in more formal way (see Algorithms \ref{alg2} and \ref{alg3}). Let us start with finding a value of $\bar z$ for which $l_C(\bar z)= 0$.

In order to do that we will use a modified version of the false position method. It is a known fact that the standard version of this procedure does not perform well, in particular it can have the rate of convergence $2/3$ in the case when the function $l_C(z)$ is strictly convex or strictly concave on the search interval (in such a situation one of the endpoints of the interval does not change). There are, however, its modifications. Probably the first of them was the one called Illinois method, having overlinear convergence rate, described in \cite{ref_DowJar}. Even better is the Anderson-Bj\"orck method \cite{ref_AndBjo}, that we will apply here. A comparison of various methods of this type can be found in \cite{ref_For}.

The method adapted for our purposes is presented as Algorithm \ref{alg02} below. Below we use the notation
$$
h_{ij}=\frac{l_C(z_j)-l_C(z_i)}{z_j-z_i}
$$
and
$$
h_i=l_C(z_i).
$$
\begin{algorithm}
\caption{Anderson-Bj\"orck Method}\label{alg02}
\begin{enumerate}
\item[Step 1:]
Assume accuracy level $\varepsilon$. Set $h_1\gets l_C(z_1)$, $h_2\gets l_C(z_2)$. Proceed to step 2.
\item[Step 2:]
If $|z_2-z_1|<\varepsilon$, then STOP, $\bar{z}=z_2$. Otherwise, go to step 3.
\item[Step 3:]
Set $h_{12}=\frac{h_2-h_1}{z_2-z_1}$, $z_3\gets z_2-l_2/h_{12}$, and $h_3\gets l_C(z_3)$. If $h_2h_3>0$, then set $h_{23}=\frac{h_3-h_2}{z_3-z_2}$, $z_2\gets z_1$, $h_2\gets h_1h_{23}/h_{12}$. Go to Step 4.
\item[Step 4:]
If $z_2<z_3$, set $z_1\gets z_2$, $h_1\gets h_2$, $z_2\gets z_3$, $h_2\gets h_3$. Otherwise, set $z_1\gets z_3$, $h_1\gets h_3$. Go back to Step 2.
\end{enumerate}
\end{algorithm}

Finally, we are able to present two solution methods. We start with the bisection method (Algorithm \ref{alg2}, see also \cite{ref_anh2012}), where the starting point is generated by the geometric means of rows of $A$.

\begin{algorithm}
\caption{Main algorithm 1}\label{alg2}
\begin{enumerate}
\item[Step 1:]
Assume the accuracy level $\varepsilon>0$. Let $v_i^\star=\left(\prod_{j=1}^n{a_{ij}}\right)^{1/n}$ and $v_i=v_i^\star/v_1^\star$ for $i=1,2,\dots,n$. Let $z=z_{\max}=G_\infty(A,v)$ and $z_{\min}=0$. Compute $w_j=\ln v_j$, $j=1,\dots,n$. Proceed to step 2.
\item[Step 2:]
If $z-z_{\min}<\varepsilon$, then STOP. Compute $v_j=\exp(w_j)$, $j=1,\dots,n$. Vector $v$ is the desired approximation of the optimal weight vector. Otherwise, go to step 3.
\item[Step 3:]
Set $z:=(z_{\max}+z_{\min} )/2$. Construct network $N(z)$. Apply the Network Simplex Method to find the shortest distances in $N$ from node $1$ to all the nodes. If there is a solution with distances equal to $w_j$, $j=1,\dots,n$, then go to step 4. Otherwise, go to step 5.
\item[Step 4:]
Set $z_{\max}\gets z$. Go back to step 2.
\item[Step 5:]
A negative cycle $C$ has been found in the network. Set $z_{\min}\gets z$. Go back to step 2.
\end{enumerate}
\end{algorithm}

We can easily prove the following proposition.

\begin{myproposition}
After a finite number of iterations, Algorithm \ref{alg2} terminates by finding an $\varepsilon$-optimal solution.
\end{myproposition}

\begin{proof}
In every step of Algorithm \ref{alg2} the value of the difference $z_{max}-z_{min}$ decreases to its half, so in a finite number of iterations we obtain the approximation of the optimal solution. More precisely, if $z_{max}^\star$ denotes the initial value of $z_{max}$, then the algorithm will stop after at most $\log_2{\left\lceil z_{max}^\star/\varepsilon\right\rceil}$ steps.
\end{proof}

We finish this section with the second proposed algorithm (Algorithm \ref{alg3}).

\begin{algorithm}
\caption{Main algorithm 2}\label{alg3}
\begin{enumerate}
\item[Step 1:]
Assume the accuracy level $\varepsilon>0$. Let $z=0$. Proceed to step 2.
\item[Step 2:]
Construct network $N(z)$. Apply the Network Simplex Method to find the shortest distances in $N$ from node $1$ to all the nodes. If there is a solution with distances equal to $w_j$, $j=1,\dots,n$, then STOP, the entries $v_j=\exp(w_j)$, $j=1,\dots,n$, form the desired approximation of the optimal weight vector. Otherwise (if there is a negative cycle $C$ in the network), go to step 3.
\item[Step 3:]
Find an $\varepsilon$-approximation $\bar{z}$ of the root of $l_C(z)=0$ using the Algorithm \ref{alg02}. Set $z\gets \bar{z}$. Go back to step 2.
\end{enumerate}
\end{algorithm}

We can easily prove the following proposition.

\begin{myproposition}
After a finite number of iterations, Algorithm \ref{alg3} terminates by finding an $\varepsilon$-optimal solution.
\end{myproposition}

\begin{proof}
The choice of the upper bound (i.e., $z_2$) as the approximation of $\bar{z}$ in Algorithm \ref{alg02} assures that a negative cycle $C$ can appear only once in the algorithm. Thus the number of steps is bounded from above by the number of cycles in the complete digraph on $n$ vertices, which is equal to
$$\sum_{k=2}^{n}{{n\choose k} (k-1)!}<(n+1)!$$
\end{proof}

\section{Uniqueness of the solution}\label{section_unique}

One of the nice features of an optimization problem is the fact that it has a unique optimal solution. In such a case there is no need to introduce additional rules in order to choose the final solution. Although the problem (\ref{problem11})-(\ref{problem12}) does not have this property in general, we are going to characterize the necessary conditions for the uniqueness of the solution.

In the case when the problem has more than one optimal solution, we are going to characterize the set of optimal solutions exactly. We also propose a way for deriving a unique optimal solution after introducing an  additional, rather obvious, criterion.

Let us consider the system (\ref{new_system1})-(\ref{new_system6}). We can rewrite it in the following way:
\begin{eqnarray}
v_i\geq \max\{a_{ij}-z,\frac{1}{a_{ji}+z}\}v_j=L_{ij}(z) v_j,1\leq i<j\leq n,\label{proof_system1}\\
v_j\geq \max\{a_{ji}-z,\frac{1}{a_{ij}+z}\} v_i=L_{ji}(z) v_i,1\leq i<j\leq n,\label{proof_system2}\\
v_j> 0,j=1,\dots, n,\label{proof_system3}\\
v_1=1.\label{proof_system4}
\end{eqnarray}

Obviously, the functions $L_{ij}(\cdot)$ are strictly decreasing and continuous for $z\geq 0$ (it follows directly from the fact that $L_{ij}(z)=\exp(-l_{ij}(z))$ for every $z\geq 0$ and from the Proposition \ref{lij_properties}). Moreover, the following is true.

\begin{myobservation}\label{not_both_binding}
For every $1\leq i<j\leq n$, $L_{ij}(z)L_{ji}(z)=1$ if and only if $z=0$.
\end{myobservation}

\begin{proof}
If $z=0$, then we have
$$
L_{ij}(z)L_{ji}(z)=\max\{a_{ij}-z,\frac{1}{a_{ji}+z}\}\cdot \max\{a_{ji}-z,\frac{1}{a_{ij}+z}\}=a_{ij}a_{ji}=1.
$$
On the other hand, let $L_{ij}(z)L_{ji}(z)=1$. Since both $L_{ij}(\cdot)$ and $L_{ji}(\cdot)$ are strictly decreasing, continuous and positive for $z\geq 0$, so is also the function $L^\star_{ij}(z)=L_{ij}(z)L_{ji}(z)$. In particular, it can take the value $1$ only at one point and we already know that it is equal to $1$ when $z=0$.
\end{proof}

The next result is as follows.

\begin{mylemma}\label{new_sol_when_not_binding}
Assume that the system (\ref{proof_system1})-(\ref{proof_system4}) has a feasible solution $v=(v_1,\dots,v_n)$ for a given $z=z_1>0$. If all the inequalities (\ref{proof_system1})-(\ref{proof_system2}) are not binding, i.e.
\begin{eqnarray}
v_i>L_{ij}(z) v_j,1\leq i<j\leq n,\label{proof2_system1}\\
v_j>L_{ji}(z) v_i,1\leq i<j\leq n,\label{proof2_system2}
\end{eqnarray}
then there exists a number $z_2<z_1$, such that $v$ is a feasible solution of the system (\ref{proof_system1})-(\ref{proof_system4}) also for $z=z_2$.
\end{mylemma}

\begin{proof}
For every $1\leq i\neq j\leq n$, and for every $0<\delta_{ij}\leq z_1$, we have $L_{ij}(z_1-\delta_{ij})>L_{ij}(z_1)$. Let us define $\delta_{ij}$ as
$$
\delta_{ij}=\max\{\delta|v_i\geq L_{ij}(z_1-\delta) v_j\}.
$$
Each $\delta_{ij}$ is well-defined, as it is bounded from above by $z_1$ and uniquely defined, as $L_{ij}(\cdot)$ is strictly decreasing and continuous. Let
$$
\hat\delta=\min\{\delta_{ij}|1\leq i\neq j\leq n\}.
$$
Observe that $\hat\delta>0$ and
$$
v_i\geq L_{ij}(z_1-\hat\delta) v_j
$$
for each $1\leq i\neq j\leq n$. Thus $v$ is a feasible solution also for $z=z_2=z_1-\hat\delta$.
\end{proof}

Now let us deal with the case when some of the inequalities (\ref{proof_system1})-(\ref{proof_system2}) are binding. Let us define the digraph $D=D(v,z)$ in the following way. Let the vertex set of $D$ be $V(D)=\{1,\dots,n\}$ and let the arcs set of $D$ be $A(D)=\{(i,j)|v_i=L_{ij}(z)v_j\}$. Then the following observation is true.

\begin{mylemma}\label{new_sol_find_not_binding}
Assume that the system (\ref{proof_system1})-(\ref{proof_system4}) has a feasible solution $v=(v_1,\dots,v_n)$ for a    given $z>0$. If there is no (directed) cycle in $D(v,z)$, then there exists a feasible solution $v^\prime=(v_1^\prime,\dots,v_n^\prime)$ of (\ref{proof_system1})-(\ref{proof_system4}) such that all the inequalities (\ref{proof_system1})-(\ref{proof_system2}) are not binding.
\end{mylemma}

\begin{proof}
If there is no cycle in $D$, then there is a vertex $i\in V(D)$ being the start vertex of all the arcs incident with it. It means in turn that there is a variable $v_i$ such that
\begin{eqnarray}
v_j>L_{ji}(z)v_i, j=1,\dots,n, j\neq i.\label{ineq_dominant}
\end{eqnarray}
The last inequalities follow from the fact that both $v_i\geq L_{ij}(z)v_j$ and $v_j\geq L_{ji}(z)v_i$ cannot be simultaneously binding when $z>0$, see Observation \ref{not_both_binding}. Now we are going to find a new solution in which no inequality where $v_i$ occurs, is binding. Let us define $\delta_{ij}, j=1,\dots,n, j\neq i$ as follows:
$$
\delta_{ij}=\frac{1}{2}\left(\frac{v_j}{L_{ji}(z)}-v_i\right).
$$
Now let
$$
v_i^\prime=v_i+\min\{\delta_{ij}|j=1,\dots,n, j\neq i\}.
$$
The solution $v^{\prime}$, where we substitute the value of $v_i$ with $v_i^\prime$, does not change any inequality where $v_i$ does not occur, changes all the equalities with $v_i$ to sharp inequalities and preserves all the inequalities (\ref{ineq_dominant}). It means that the new solution is feasible, and the graph $D(v^{\prime},z)$ is a proper subgraph of $D(v,z)$, i.e. $A(D(v^{\prime},z))\subsetneq A(D(v,z))$. To be more specific, all the arcs incident with $i$ have been removed from $D$. Observe also that $D(v^{\prime},z)$ is acyclic. Now if $D(v^{\prime},z)$ is still nonempty, then we substitute $v:=v^{\prime}$ and repeat the process. Of course after at most $n-1$ steps we finally obtain a solution $v$ for which $A(D(v,z))=\emptyset$, i.e. there are no binding equalities among (\ref{proof_system1})-(\ref{proof_system2}). This finishes the proof.
\end{proof}

The two lemmas presented above lead us to the following conclusion.

\begin{mycorollary}\label{no_cycle_no_optimum}
Assume that the system (\ref{proof_system1})-(\ref{proof_system4}) has a feasible solution $v=(v_1,\dots,v_n)$ for given $z>0$. Assume that there is no (directed) cycle in $D(v,z)$. Then $v$ is not an optimal solution of the problem (\ref{problem11})-(\ref{problem12}).
\end{mycorollary}

Now we will focus on the case, when there is a cycle in $D(v,z)$. Let us denote with $D^\prime=D^\prime(v,z)$ the subgraph of $D$ such that $V(D^\prime)=V(D)$ and an arc of $D$ is also arc in $D^\prime$ if and only if it belongs to a (directed) cycle in $D$. In other words, $D^\prime$ consists of all the strongly connected components of $D$. The following is true.

\begin{mylemma}\label{many_solutions_if_cycle}
Assume that the system (\ref{proof_system1})-(\ref{proof_system4}) has a feasible solution $v=(v_1,\dots,v_n)$ for a    given $z>0$. Assume that there is a (directed) cycle in $D(v,z)$. Assume that $D^\prime$ has $c$ components. Then:
\begin{enumerate}
\item
$v$ is optimal solution of (\ref{problem11})-(\ref{problem12}).
\item
The set of the solutions of (\ref{proof_system1})-(\ref{proof_system4}) is a convex polytope of dimension $d=c-1$.
\end{enumerate}
\end{mylemma}

\begin{proof}
Assume that one of the cycles consists of the vertices $i_1,i_2,\dots,i_s$. Then we have
\begin{eqnarray*}
v_{i_1}= L_{i_1i_2}(z)v_{i_2},\\
v_{i_2}= L_{i_2i_3}(z)v_{i_3},\\
\vdots\quad\quad\quad\quad\quad \\
v_{i_{s-1}}= L_{i_{s-1}i_s}(z)v_{i_s},\\
v_{i_s}= L_{i_si_1}(z)v_{i_1}.\\
\end{eqnarray*}
It means in particular that
$$
\prod_{j=1}^{s}{L_{i_ji_{j+1\imod s}}(z)}=1,
$$
and consequently $z$ has reached its minimal value as the functions $L_{ij}(\cdot)$ are strictly decreasing and so any decrease of $z$ would make this product greater than $1$, while if the inequalities (\ref{proof_system1})-(\ref{proof_system2}) are satisfied, then
\begin{eqnarray*}
v_{i_1}\geq L_{i_1i_2}(z)v_{i_2},\\
v_{i_2}\geq L_{i_2i_3}(z)v_{i_3},\\
\vdots\quad\quad\quad\quad\quad \\
v_{i_{s-1}}\geq L_{i_{s-1}i_s}(z)v_{i_s},\\
v_{i_s}\geq L_{i_si_1}(z)v_{i_1},\\
\end{eqnarray*}
and thus
$$
\prod_{j=1}^{s}{L_{i_ji_{j+1\imod s}}(z)}\leq 1.
$$
This means that $v$ is an optimum of (\ref{problem11})-(\ref{problem12}).
The set of optimal solutions of (\ref{problem11})-(\ref{problem12}) is a convex polytope as it is exactly the set of feasible solutions of the linear system (\ref{proof_system1})-(\ref{proof_system4}).
On the other hand, each variable $v_j$, $j\in\{i_1,i_2,\dots,i_{s-1}\}$ may be expressed as $c_{i_1j}(z)v_{i_1}$, where $c_{i_1j}(z)$ depends only on $z$ i.e., is a constant when $z$ is fixed. As this reasoning is true for every cycle in $D^\prime(v,z)$, it follows that for every component $C_t$ of $D^\prime(v,z)$ with $s_t$ vertices, $t=1,\dots, c$, $(s_t-1)$ variables may be expressed as a chosen $s_t^{th}$ variable multiplied by a constant. On the other hand, the reasoning similar as in the proof of Lemma \ref{new_sol_find_not_binding} shows that it is not the case for the remaining variables, corresponding with the vertices not belonging to any cycle of $D(v,c)$.  This means that every component $C_t$ of $D^\prime(v,z)$ reduces the dimension of the set of solutions by $(s_t-1)$ (in particular, if it is a trivial component, it does not reduce the dimension). Since one of the variables is fixed ($v_1=1$), we obtain
$$
d=n-1-\sum_{t=1}^{c}{(s_t-1)}=n-1-\sum_{t=1}^{c}{s_t}+\sum_{t=1}^{c}{1}=n-1-n+c=c-1.
$$
\end{proof}

\begin{mycorollary}\label{all_in_cycles_one_solution}
Assume that the system (\ref{proof_system1})-(\ref{proof_system4}) has a feasible solution $v=(v_1,\dots,v_n)$. Assume that there is a (directed) cycle in $D(v,z)$. If $D^\prime=D^\prime(v,z)$ is connected, then $v$ is the only optimal solution of (\ref{problem11})-(\ref{problem12}).
\end{mycorollary}

\begin{proof}
If $z>0$, then we have $c=1$, what implies $d=0$, so the set of the optimal solutions is a point. If $z=0$, then there is exactly one solution $v_j=a_{j1}$ and $A$ is consistent.
\end{proof}

Note that the last corollary can be also deduced from the results obtained by Blanquero, Carrizosa and Conde in \cite{ref_BlaCarCon} (Corollary 11).

\begin{mycorollary}
The set of optimal solutions of (\ref{problem11})-(\ref{problem12}) is a convex polytope of dimension at most $n-3$.
\end{mycorollary}

\begin{proof}
If $z>0$, then any cycle in $D^\prime(v,z)$ must have at least $3$ vertices (see Observation \ref{not_both_binding}). This implies that each component of $D^\prime(v,z)$ having more than one vertex must consist of at least $3$ vertices and from Corollary \ref{no_cycle_no_optimum} it follows that there must be at least one such component when $z>0$. From the proof of Lemma \ref{many_solutions_if_cycle} it follows that
$$
d=n-1-\sum_{t=1}^{c}{(s_t-1)}\leq n-1-2= n-3.
$$
\end{proof}

\begin{mycorollary}
The problem (\ref{problem11})-(\ref{problem12}) has one solution if $n\leq 3$.
\end{mycorollary}

\begin{proof}
From the previous corollary it follows that in such a case $d\leq 0$, what implies that $d=0$ and the set of optimal solutions is a point.
\end{proof}



The main result of this section is the following theorem, being an immediate consequence of the Corollaries \ref{no_cycle_no_optimum} and \ref{all_in_cycles_one_solution}.

\begin{mytheorem}\label{cond_sol_unique}
The problem (\ref{problem11})-(\ref{problem12}) has a unique optimal solution if and only if for some $z\ge 0$ there exists a feasible solution $v=(v_1,\dots,v_n)$ resulting with the objective value $z$, such that all the (directed) cycles of $D(v,z)$ induce a connected subgraph of $D(v,z)$ that covers all the vertices of $D(v,z)$. In such a case, $z$ is the optimal value of the objective function.
\end{mytheorem}

The question is, when the assumption from the above theorem is satisfied. Of course we are looking for some conditions that could be imposed in the real life. One could ask if the transitivity of the preference matrix will be enough, i.e if the sufficient condition could be
$$
a_{ij}>1 \wedge a_{jk}>1\Rightarrow a_{ik}>1.
$$
Unfortunately this is not the case. Even more restrictive condition:
$$
a_{ij}>1 \wedge a_{jk}>1\Rightarrow a_{ik}>\max\{a_{ij},a_{jk}\}
$$
is not sufficient, as the following example shows:
$$
A=\left[
\begin{array}{cccc}
1 & 3 & 2/7 & 11/10\\
1/3 & 1 & 1/7 & 9/10\\
7/2 & 7 & 1 & 5\\
10/11 & 10/9 & 1/5 & 1
\end{array}
\right]
$$
The set of optimal solutions is a one-dimensional segment with the endpoints $(1,0.4,3,0.625)$ and $(1,0.4,3,0.6(4))$, and the optimal value is $z=0.5$.

As we have already observed, it is not obvious that the problem (\ref{problem11})-(\ref{problem12}) has unique solution. However, even if the set of the solutions is infinite, we can still propose a nice comparison method that allows us to choose the best one.

Our main goal is to minimize the maximum over all the deviations $|a_{ij}-v_i/v_j|$. If this maximum equals to $z$, it does not matter, from the mathematical point of view, whether the other deviations are equal to $z$ or less than $z$ and what are their exact values. However, if we include the psychological aspects, then of course the solution where the deviations not equal to $z$ are as small as possible, should be preferred. The method presented in the next section satisfies this condition. Moreover, it leads us to a Pareto-optimal solution.

\section{New method of deriving the priorities}

Assume that we found an optimal solution $(v,z)$ of the problem (\ref{problem21})-(\ref{problem22}). From the previous section we know that there is at least one directed cycle of length at least $3$ in the graph $D^\prime(v,z)$. Moreover, for each non-trivial component $C_i$ of $D^\prime(v,z)$, by choosing one of its vertices (reference vertex), say $i$ and the corresponding variable $v_i$, we can express each variable $v_k$, $k\in C_i$ in the form

\begin{equation}
v_k=c^\star_kv_i,\label{formula_constant}
\end{equation}

\noindent{}where $c^\star_k$ is a constant. This is equivalent to expressing the variables $w_k$ as

\begin{equation}
w_k= c_k+w_i\label{formula_constant2}.
\end{equation}

Thus it is possible to substitute each of the variables $w_k$ by an expression with $c_k$ and finally remove a part of the inequalities in the following way. We choose one vertex from each strongly connected component of $D(v,z)$. These vertices are the vertices of a new, reduced network. Now observe that the cost of the edges inside the former components do not matter, so we can remove all the inequalities of the form $w_q-w_p\leq l_{pq}(z)$ for $p,q\in C_i$, $p\neq q$. Then, let us choose two vertices $i$ and $j$ of the reduced network. Let $p\in C_i$ and $q\in C_j$. The inequality
$$
w_q-w_p\leq l_{pq}(z)
$$
can be rewritten as
$$
w_j+c_q-w_i-c_p\leq l_{pq}(z)
$$
or
$$
w_j-w_i\leq l_{pq}(z)+c_p-c_q.
$$
The last inequality must be satisfied for all $p\in C_i$ and $q\in C_j$, so we can write it as
$$
w_j-w_i\leq \min\{l_{pq}(z)+c_p-c_q|p\in C_i, q\in C_j\}=l^\star_{ij}(z).
$$
The functions $l^\star_{ij}(z)$ are cost functions in the new network. Unfortunately, $l_{ij}^\star(z)=l_{pq}(z)+c_p-c_q$ does not mean that $l_{ji}^\star(z)=l_{qp}(z)+c_q-c_p$, which have some implications for their properties. It is straightforward to see that $l^\star_{ij}(z)$ are not necessarily differentiable and may have more than one inflection point, but still are strictly increasing. One of the consequences is also that the set of optimal solutions of the reduced network has dimension at most $n-2$, not $n-3$, since this time there can be zero-length cycles on two nodes even if $z>0$. The consequence is that the first reduction of the network reduces the number of variables (and vertices in the network) by at least $2$ and each other by at least $1$. This process ends when there are no zero-valued cycles or when the number of vertices is $2$ (in both cases there is exactly one solution of the problem). This leads us to Algorithm \ref{alg_unique}.

\begin{algorithm}
\caption{Choosing the unique optimum}\label{alg_unique}
\begin{enumerate}
\item[Step 1:]
Denote the initial problem with $P$. Let $N$ be the network corresponding with $P$. Initialize the list $LC$ (the list of the subsets of vertices of $N$ forming a partition of $V(N)$). For each vertex $i$ of $N$, add the set $C_i=\{i\}$ to $LC$. For each existing set $C_i$ (some of them will be removed later), $i$ will be the reference vertex. Solve $P$ using $N$. Let the solution be $v$ and let the resulting objective value be $z$. Go to step $2$.
\item[Step 2:]
If the subgraph $D=D(v,z)$ of $N$ is strongly connected, then STOP, $v$ is the desired solution, go to step $4$. Otherwise, go to step $3$.
\item[Step 3:]
For each strongly connected component $C$ of $D(v,z)$, choose any vertex $i\in C$. For each vertex $j\in C, j\neq i$ add all the vertices $k\in C_j$ to the set $C_i$ and remove $C_j$ from $LC$ ($j$ is the reference vertex of some set $C_j$ from the list $LC$). Compute new values of $c_k$ for $k\in C_i, k\neq i$ by applying formula (\ref{formula_constant2}). Denote new, reduced problem with $P$ and corresponding network with $N$. Solve $P$ using $N$. Let the solution be $v$ and let the resulting objective value be $z$. Go back to step $2$.
\item[Step 4:]
Having the unequivocal solution of the last instance of the problem $P$, find the solution of the initial problem by using the formula (\ref{formula_constant}).
\end{enumerate}
\end{algorithm}

Observe that in the step $3$, at the first iteration, the dimension of the problem (\ref{problem11})-(\ref{problem12}) decreases by at least $2$, and at each other iteration  by at least $1$. This means that the desired unequivocal optimum will be found after at most $n-2$ iterations (i.e. after solving at most $n-2$ instances of LWAE problem). Moreover, the following is true.

\begin{mytheorem}
The solution obtained by the Algorithm \ref{alg_unique} is Pareto-optimal.
\end{mytheorem}

\begin{proof}
If the above procedure ends in one step, it follows from the Theorem \ref{cond_sol_unique}, that the obtained solution is unique and in consequence Pareto-optimal. If it takes more steps, the thesis follows from the following fact (see e.g. \cite{ref_BlaCarCon}, Proof of Theorem 1).
\begin{myfact}
Given a solution $v$ of (\ref{vector_problem}), let $\varepsilon_{ij}=|v_i/v_j-a_{ij}|$. Then $v$ is Pareto-optimal if and only if for every pair $k\neq l, k,l\in \{1,2,\dots,n\}$, $v$ is the solution of the problem:
\begin{eqnarray}
\min{\varepsilon_{kl}}\\
|v_i/v_j-a_{ij}|\leq\varepsilon_{ij}, (i,j)\neq (k,l)\\
v_i>0,i=1,\dots,n.
\end{eqnarray}
\end{myfact}
\end{proof}

\section{Computational experiments}

The algorithm has been implemented in Java and and tested for a number of randomly generated problems. The assumed accuracy level was $\varepsilon=0.000001$. The application has been tested on the PC with Intel(R) Core(TM) i7-2670 QM(2.20GHz). We tested problems with $n=10$ and $n=20$ (one rather should not expect the problems of size greater than $10$ in the real life, but this allowed us to examine the efficiency of our algorithms). We did not impose any conditions on the inconsistency of the initial matrix, so the elements of $A$ were chosen uniformly at random from the set $\{1,2,\dots,a_{max}\}$, where $a_{max}\in\{3,5,10\}$, and then the chosen elements and its reciprocal were put on two sides of the diagonal in the random order. In every case $100$ problems have been solved (that gives the total number of $600$ test problems). The information on the running times of Algorithm \ref{alg2} (Time1, in seconds), running times of Algorithm \ref{alg3} (Time2, in seconds) and the number of LWAE subproblems solved (\#LW) are given in Table \ref{table1} (AVG - average time in seconds, DEV - standard deviation, MIN - minimum time in the sample, MAX - maximum time in the sample).

\begin{table}[ht]
\caption{Results of tests}\label{table1}
\begin{center}
\begin{tabular}{|c|c|r|r|r|r|r|r|}
\hline
\multirow{2}{*}{$a_{max}$}	&\multirow{2}{*}{Parameter}	&\multicolumn{3}{|c|}{$n=10$}	&\multicolumn{3}{|c|}{$n=20$}	\\
\cline{3-8}
								&&Time1	&Time2	&\#LW	&Time1	&Time2	&\#LW	\\
\hline
\multirow{4}{*}{$3$}	&AVG	&0.0041	&0.0024	&3.41	&0.0064	&0.0021	&1.86	\\
						&DEV	&0.0019	&0.0012	&1.07	&0.0013	&0.0010	&0.73	\\
						&MIN	&0.0016	&0.0005	&1		&0.0047	&0.0005	&1		\\
						&MAX	&0.0125	&0.0078	&6		&0.0115	&0.0055	&5		\\
\hline
\multirow{4}{*}{$5$}	&AVG	&0.0038	&0.0021	&4.43	&0.0110	&0.0054	&4.49	\\
						&DEV	&0.0015	&0.0007	&0.96	&0.0037	&0.0024	&1.87	\\
						&MIN	&0.0016	&0.0011	&2		&0.0059	&0.0014	&1		\\
						&MAX	&0.0095	&0.0045	&7		&0.0228	&0.0117	&10		\\
\hline	
\multirow{4}{*}{$10$}	&AVG	&0.0063	&0.0044	&4.80	&0.0320	&0.0169	&8.67	\\
						&DEV	&0.0022	&0.0014	&1.05	&0.0108	&0.0048	&1.61	\\
						&MIN	&0.0022	&0.0011	&2		&0.0125	&0.0062	&5		\\
						&MAX	&0.0172	&0.0086	&7		&0.0578	&0.0360	&13		\\
					
\hline
\end{tabular}
\end{center}
\end{table}

As we can see, in all the cases the running times are much less than one second, which is acceptable in real life applications. The times in general increase with $a_{max}$ and $n$, although there are some exceptions. The most interesting one is probably the fact that Algorithm \ref{alg_unique} with Algorithm \ref{alg3} used for the LWAE subproblems ran faster for bigger problems where $a_{max}=3$. The explanation of this phenomenon is probably the fact that in the presence of so little possible values of the entries of $A$, the bigger matrix became more consistent (the objective function became more "flat"), what resulted in smaller number of necessary runs of Algorithm \ref{alg3}.

One can also observe that it is more efficient to use Algorithm \ref{alg3} (successive canceling of negative cycles) to solve the subproblems - in all cases this algorithm behaves better than Algorithm \ref{alg2} (bisection).

\section{Conclusion}

In the paper we provide a new method of deriving priorities from an inconsistent Pairwise Comparison Matrix. Our method produces a Pareto-optimal solution very quickly because of using the logarithmic transformation and network algorithms. The running times are so small that the new method is competitive against other methods of finding the Pareto-efficient solutions of (\ref{vector_problem}).

An interesting open problem is whether the graph theoretic approach applied successfully in this paper can be  preserved for the case of $G_1(A,v)$. We leave it as an open question for future research.

\section*{Acknowledgments}
The research of the first author was part of the project "Nonlinear optimization in chosen economical applications". The project was financed by the National Science Center grant awarded on the basis of the decision number DEC-2011/01/D/HS4/03543. The research of the second author was supported in part by OTKA grant 111797.

\end{document}